\documentclass[onecolumn,draftclsnofoot]{IEEEtran}
\usepackage{lipsum}
\usepackage{amsfonts}
\usepackage{graphicx}
\usepackage{epstopdf}
\usepackage[sort,noadjust]{cite}
\usepackage{amsopn}

\usepackage{amsmath} 
\usepackage{amsthm}
\usepackage{amssymb}  
\usepackage{mathptmx} 

\usepackage{bbm}

\usepackage{algorithm}
\usepackage{algorithmic}

\usepackage{pgfplots}
\pgfplotsset{compat=newest}

\usepackage{tikz}
\usepackage{tkz-euclide}

\usepackage{enumerate}

\usepackage{mathrsfs}
\definecolor{ao(english)}{rgb}{0.0, 0.5, 0.0}
\usepackage[colorlinks, citecolor = {ao(english)}, linkcolor = {ao(english)}]{hyperref} 

\usepackage[capitalize,nameinlink]{cleveref}[0.19]

\usepackage{siunitx}

\usepackage{subcaption}
\usetikzlibrary{arrows,shapes,trees,calc,positioning,patterns,decorations.pathmorphing,decorations.markings,backgrounds}
\usetikzlibrary{matrix}
\usepgfplotslibrary{groupplots}

\newtheorem{thm}{Theorem}
\crefname{thm}{Theorem}{Theorems}

\newtheorem{prop}{Proposition}
\crefname{prop}{Proposition}{Propositions}

\newtheorem{lem}{Lemma}
\crefname{lem}{Lemma}{Lemmas}

\newtheorem{cor}{Corollary}
\crefname{cor}{Corollary}{Corollaries}

\crefname{rem}{Remark}{Remark}

\crefname{ass}{Assumption}{Assumption}

\crefname{conj}{Conjecture}{Conjectures}

\crefname{defn}{Definition}{Definitions}

\newtheorem{prob}{Problem}
\crefname{prob}{Problem}{Problems}

\crefname{appl}{Application}{Applications}

\crefname{algorithm}{Algorithm}{Algorithms}
\crefname{paper}{Paper}{Papers}
\crefname{figure}{Figure}{Figures}
\crefname{section}{Section}{Sections}
\Crefname{section}{Section}{Sections}

\usepackage{dsfont}
\let\mathbb=\mathds

\newcommand{\Rmn}{\mathbb{R}^{n \times m}}

\newcommand{\Rnn}{\mathbb{R}^{n \times n}}

\newcommand{\conv}{\textnormal{conv}}
\newcommand{\rk}{\textnormal{rank}}
\newcommand{\svd}{\textnormal{svd}}
\newcommand{\diag}{\textnormal{diag}}

\newcommand{\argmin}{\operatornamewithlimits{argmin}}

\newcommand{\trace}{\textnormal{trace}}

\newcommand{\minmn}{\min \{ m,n \}}

\newcommand{\opts}{\star}

\newcommand{\normrast}[1]{\| #1 \|_{r\ast}}
\newcommand{\normr}[1]{ \| #1 \|_{r}}

\newcommand{\transp}{\mathsf{T}}

\newcommand{\funcdom}{\Rmn \to \mathbb{R} \cup \lbrace \infty \rbrace}

\definecolor{ao(english)}{rgb}{0.0, 0.5, 0.0}

\colorlet{FigColor1}{red}
\colorlet{FigColor1b}{orange}

\colorlet{FigColor2}{blue}
\colorlet{FigColor2b}{magenta}

\colorlet{FigColor4}{ao(english)}

\colorlet{FigColor3}{gray}
\colorlet{FigColor3b}{black}

\pgfplotsset{every axis plot/.append style={line width=1.4pt}}
\definecolor{bluebell}{rgb}{0.74, 0.83, 0.9}
\definecolor{airforceblue}{rgb}{0.36, 0.54, 0.66}

\crefformat{equation}{\textup{#2(#1)#3}}
\crefrangeformat{equation}{\textup{#3(#1)#4--#5(#2)#6}}
\crefmultiformat{equation}{\textup{#2(#1)#3}}{ and \textup{#2(#1)#3}}
{, \textup{#2(#1)#3}}{, and \textup{#2(#1)#3}}
\crefrangemultiformat{equation}{\textup{#3(#1)#4--#5(#2)#6}}%
{ and \textup{#3(#1)#4--#5(#2)#6}}{, \textup{#3(#1)#4--#5(#2)#6}}{, and \textup{#3(#1)#4--#5(#2)#6}}

\Crefformat{equation}{#2Equation~\textup{(#1)}#3}
\Crefrangeformat{equation}{Equations~\textup{#3(#1)#4--#5(#2)#6}}
\Crefmultiformat{equation}{Equations~\textup{#2(#1)#3}}{ and \textup{#2(#1)#3}}
{, \textup{#2(#1)#3}}{, and \textup{#2(#1)#3}}
\Crefrangemultiformat{equation}{Equations~\textup{#3(#1)#4--#5(#2)#6}}%
{ and \textup{#3(#1)#4--#5(#2)#6}}{, \textup{#3(#1)#4--#5(#2)#6}}{, and \textup{#3(#1)#4--#5(#2)#6}}

\crefdefaultlabelformat{#2\textup{#1}#3}

\begin{document}
%
\title{Low-Rank Optimization with Convex Constraints}
%
%
%

\author{Christian~Grussler,~\IEEEmembership{}
        Anders~Rantzer,~\IEEEmembership{}
        and~Pontus Giselsson.~\IEEEmembership{}
}

\maketitle

\begin{abstract}
The problem of low-rank approximation with convex constraints, which appears in data analysis, system identification, model order reduction, low-order controller design and low-complexity modelling is considered. Given a matrix, the objective is to find a low-rank approximation that meets rank and convex constraints, while minimizing the distance to the matrix in the squared Frobenius~norm. 
In many situations, this non-convex problem is convexified by nuclear norm regularization. 
However, we will see that the approximations obtained by this method may be far from optimal. In this paper, we propose an alternative convex relaxation that uses the convex envelope of the squared Frobenius~norm and the rank constraint. With this approach, easily verifiable conditions are obtained under which the solutions to the convex relaxation and the original non-convex problem coincide. 
An SDP representation of the convex envelope is derived, which allows us to apply this approach to several known problems. Our example on optimal low-rank Hankel approximation/model reduction illustrates that the proposed convex relaxation performs consistently better than nuclear~norm regularization and may outperform balanced truncation.
\end{abstract}


%
\IEEEpeerreviewmaketitle

\section{Introduction}
Optimization problems with a low-rank (sparsity) constraint have received considerable attention in data driven areas such as image analysis, multivariate linear regression and matrix completion (see,~e.g.~\cite{hastie2015statistical,recht2010guaranteed,velu2013multivariate,vidal2016generalized,elden2007matrix}), as well as many control subjects such as model order reduction, low order/sparse controller design, low complexity modelling, system identification, etc.~\cite{antoulas1997approximation,fazel2001rank,ankelhed2011design,zoltowski2014sparsity,zare2014low,grussler2016covariance,miller2012identification,hjalmarsson2012ident,liu2013nuclear,liu2010interior,zorzi2015factor,zorzi2016identification}.
This is because low-rank approximations allow us to study high dimensional (complex) problems in lower dimensional (simpler) domains. For example, the low-rank approximation of a Hankel operator or matrix requires a smaller number of equations to describe a dynamical system or controller, see,~e.g.~ \cite{glover1984all,ankelhed2011design,antoulas1997approximation,kung1987identification}. 

For unitarily invariant~norms an optimal low-rank approximation can be found by performing a singular value decomposition (SVD). Unfortunately, these approximations usually do not fulfil desired structural constraints such as element-wise nonnegativity, Hankel structure or prescribed entries \cite{chu2003structured,recht2010guaranteed,velu2013multivariate,miller2012identification}. Only in a few cases, an explicit solution to the constrained low-rank approximation problem is known \cite{antoulas2005approximation,velu2013multivariate,glover1984all}. For this reason, other concepts based on convex optimization have been developed~\cite{fazel2001rank,larsson2016convex,recht2010guaranteed,chandrasekaran2012convex,bach2012optimization}. Many of them rely on nuclear norm regularization, which for particular constraints and assumptions can guarantee a minimum rank solution~\cite{candes2009exact,recht2010guaranteed}. As a result, this technique (see~\cite{hjalmarsson2012ident,zare2016color,liu2013nuclear,fazel2001rank}) and its extensions (see~\cite{zorzi2015factor,zorzi2016identification}) has become a standard tool within control. Nevertheless, it is demonstrated here that nuclear regularization may be far from obtaining the optimal solution to the underlying non-convex problem. 

In this work, we study the optimal Frobenius norm low-rank approximation problem with a prescribed target rank and convex constraints (see~\cref{prob:nonneg_opt_low}). We provide an expression for the convex envelope (or equivalently the bi-conjugate) of 
\begin{align*}
f(M) = \|N-M\|_F^2+\chi_{\rk(M)\leq r}(M),
\end{align*}
where $N$ is a known data matrix and $\chi_{\rk(M)\leq r}(M)$ is the indicator function that allows for matrices of rank at most $r$. This is used to extend our work in~\cite{grussler2015optimal} to a more general setting and to provide further analysis. 

One formulation of this convex envelope has recently been presented in \cite{larsson2016convex}. In this work, we show how the bi-conjugate can be expressed very neatly in terms of the dual norm of the $r$-norm (the $\ell_2$ norm of the $r$ largest singular values). This dual norm is referred to as the $r\ast$ norm. A convex relaxation to problems involving $f$ with an additional constraint then naturally arises from the convex envelope of $f$. We provide guarantees and an example for when a globally optimal solution to our non-convex problem involving $f$ can be found by the proposed convex relaxation. 
We also show how to construct $r\ast$ norms for non-integer valued $r$. This gives rise to other convex relaxations in which the $r$ can be used as a regularization parameter to trade-off rank and data misfit in the solution. 

Further, an SDP-representation of the convex envelope is presented, which allows us to compute solutions to problems with SDP-representable constraints. This is particularly useful if the problem is of medium size (see~e.g.~\cite{ankelhed2011design}), but where it may be tedious to handle a large number of constraints with first order methods~\cite{combettes2011proximal}. Nevertheless, there are several important cases, e.g. Hankel structure, where first order methods can be used to solve problems of large size (see~\cite{grussler2017PhD,grussler2016lowrank,grussler2017local} and \cite{grussler2018github} for available implementations). 

The paper is organized as follows. In \Cref{sec:low-rank}, we introduce some definitions, recap the unconstrained low-rank approximation problem and define our main problem. Our main approach is derived and discussed in \Cref{sec:main}. Extensions of our approach to non-integer valued $r$ are discussed in \Cref{subsec:reg-interp} and corresponding SDP-representations are derived in~\Cref{sec:comp}.
In~Section \ref{sec:hankel}, an application to the open problem of Hankel structure optimal low-rank approximation \cite{blondel2012open} is presented. These approximations are used to construct reduced order models and to compare their performance with balanced truncation~\cite{antoulas2005approximation}. Finally, we draw conclusions and discuss future research in~\Cref{sec:conclusion}.
\section{Background}
\label{sec:low-rank}
\subsection{Notations}
The following notations for real matrices $X=(x_{ij}) \in \Rmn$ is used throughout this paper. Without loss of generality, it is assumed that $n \leq m$. Submatrices of $X$ are denoted by $$X_{(p:q,s:t)} := (x_{ij})_{\substack{p \leq i \leq q,\ s \leq j \leq t}} \in \mathbb{R}^{p-q+1 \times s-t+1}.$$ If $X = X^{\transp}$ is positive definite (semi-definite) we use the notation $X \succ 0$ ($X \succeq 0$). We also use these notations to describe the relation between two matrices, e.g. $A \succeq B$ means $A-B \succeq 0$.

The non-increasingly ordered singular values of $X \in \Rmn$ are denoted by $\sigma_1(X) \geq \dots \geq \sigma_{n}(X),$ counted with multiplicity.  The Frobenius inner-product for $X,Y \in \Rmn$ is defined as $$\langle X , Y \rangle := \sum_{i=1}^{m} \sum_{j=1}^{n} x_{ij} y_{ij} = \trace(X^{\transp}Y).$$ Correspondingly, the Frobenius~norm is given by $$\|X\|_F := \sqrt{\sum_{i=1}^{n}\sum_{j=1}^{n} x_{ij}^2} =  \sqrt{\sum_{i=1}^{m} \sigma_i^2(X)}.$$ 
The Frobenius~norm is so-called unitarily invariant, i.e. $\|UXV\|_F = \|X\|_F$ for all unitary matrices $U$ and $V$. The \emph{pseudo-inverse} of $X$ is denoted by $X^\dagger$ (see~e.g.~\cite{horn2012matrix}).

For a function $f: \Rmn \to \mathbb{R} \cup \lbrace \infty \rbrace$ that is linearly minorized, i.e. there exists $X \in \Rmn$ with $f(M) \geq \langle M,X \rangle$ for all $M\in \Rmn$, the \emph{conjugate function} $f^\ast$ is defined as
\begin{align*}
f^\ast(D) := \sup_{M \in \Rmn}  [\langle D,M\rangle - f(M)]
\end{align*}
for all $D\in \Rmn$. The \emph{bi-conjugate function} of $f$ is given by $f^{\ast \ast} := ({f^\ast})^\ast$. It is well-known that $f^\ast$ and $f^{\ast \ast}$ are convex (see~\cite{hiriart2013convex}). Moreover, $f(M) \geq f^{\ast \ast}(M)$ for all $M \in \Rmn$. In fact, $f^{\ast \ast}$ is the largest convex minorizer of $f$ (see~\cite[Theorem~X.1.3.5]{hiriart1996convex2}), because it is the point-wise supremum of all affine functions majorized by $f$ (see~\cref{fig:conj_biconj}). 
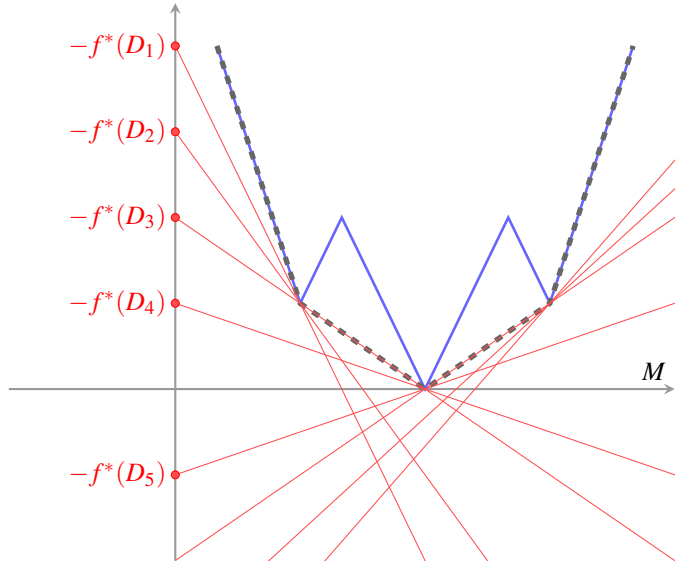
\begin{figure}[t]
	\centering
\begin{tikzpicture}
\begin{axis}[ticks=none,
xmin= -2.0, xmax=6,
ymin=-2.0, ymax=4.5,
xticklabel style = black!40,
xlabel = {$M$},
axis lines=center,
axis line style = thick,
axis line style = black!40,
legend style={at={(0.7,.95)},anchor=north east},
mark = none,
height = 9 cm]
\addplot[color=blue!60, 
line width = 1]
coordinates {
	(0.5,4)(1.5,1)(2,2)(3,0)(4,2)(4.5,1)(5.5,4)
}; \label{biconj:f}

\addplot[color=black!60, dashed, line width = 2]
coordinates {
	(0.5,4)(1.5,1)(3,0)(4.5,1)(5.5,4)
}; \label{biconj:biconj}

\addplot[red!70, domain=0:5, samples = 100,thin] (1.5*x,4-3*x); \label{biconj:sub}
\draw[color=red, fill = red!70] (axis cs:0,4) circle (1.5 pt) node[left]{$-f^\ast(D_1)$};

\addplot[red!70, domain=0:5, samples = 100,thin] (1.5*x,3-2*x);
\draw[color=red, fill = red!70] (axis cs:0,3) circle (1.5 pt) node[left]{$-f^\ast(D_2)$};

\addplot[red!70, domain=0:5, samples = 100,thin] (1.5*x,2-1*x);
\draw[color=red, fill = red!70] (axis cs:0,2) circle (1.5 pt) node[left]{$-f^\ast(D_3)$};

\addplot[red!70, domain=0:5, samples = 100,thin] (3*x,1-x);
\draw[color=red, fill = red!70] (axis cs:0,1) circle (1.5 pt) node[left]{$-f^\ast(D_4)$};

\addplot[red!70, domain=0:5, samples = 100,thin] (3*x,-1+x);
\draw[color=red, fill = red!70] (axis cs:0,-1) circle (1.5 pt) node[left]{$-f^\ast(D_5)$};

\addplot[red!70, domain=0:5, samples = 100,thin] (3*x,-2+2*x);

\addplot[red!70, domain=0:5, samples = 100,thin] (4.5*x,-3+4*x);

\addplot[red!70, domain=0:5, samples = 100,thin] (4.5*x,-4+5*x);

\end{axis}

\end{tikzpicture}
	\caption{Schematic plot of \ref{biconj:f} $f(M)$, \ref{biconj:biconj} $f^{\ast\ast}(M)$ and \ref{biconj:sub} tangents through $-f^\ast(D_i)$.}\label{fig:conj_biconj}
\end{figure} 

Finally, if $S \subset \Rmn$ and $f: \Rmn \to \mathbb{R} \cup \lbrace \infty \rbrace$, then $\argmin_S f$ denotes the \emph{set of minimizers} of $f$ over S. We write $x^\opts = \argmin_S f$, if $\argmin_S f = \lbrace x^\opts \rbrace$ is a singleton. Further, we use $\conv(S)$ to denote the \emph{convex hull} of $S$.

\subsection{Problem}
Let us turn to the underlying problem of this work. We start with the traditional optimal low-rank approximation problem in $\Rmn$, which is formulated as follows. Given $N \in \Rmn$ and $r \in \{1,\dots,n\}$, find a solution $M^\opts \in \Rmn$ to
\begin{equation}
\label{prob:trad_opt_low}
\begin{aligned}
& {\textnormal{minimize}}
& & \frac{1}{2}\|N-M\|_F^2\\
& \textnormal{subject to}
& & \rk(M) \leq r 
\end{aligned}
\end{equation}
In case of the Hilbert-Schmidt~norm, the natural operator generalization of the Frobenius-norm, this problem has been solved by Schmidt (see~\cite{stewart1990matrix}). The result is stated next.
\begin{prop}
	Let $N \in \Rmn$ and $r \in \{1,\dots ,n\}$. Then,
	$$\min_{\stackrel{M \in \Rmn}{{\rk(M) \leq r}}} \|N - M\|_F = \|\diag(\sigma_{r+1}(N),\dots,\sigma_{n}(N) ) \|_F.$$ \label{prop:Schmidt-Mirsky}
\end{prop}
All solutions to \cref{prob:trad_opt_low} are given by 
\begin{align*}
\small
\svd_r(N) := \left\{\sum_{i=1}^{r} \sigma_i(N) u_iv_i^{\transp}: N = \sum_{i=1}^{n} \sigma_i(N) u_iv_i^{\transp} \text{ is SVD of } N \right\},
\end{align*}
\textnormal
and each element in $\svd_r(N)$ is refered to as a \emph{standard SVD-approximation} of $N$. If $\sigma_r(N) = \sigma_{r+1}(N)$, then $\svd_r(N)$ contains infinitely many such solutions, because $\{u_r,u_{r+1}\}$ and $\{v_r,v_{r+1}\}$ are not uniquely determined. Otherwise, $\svd_r(N)$ is a singleton, i.e., if $\sigma_r(N) \neq \sigma_{r+1}(N)$ or $\sigma_r(N) = 0$, and we simply write $\svd_r(N)$ for the unique solution to \cref{prob:trad_opt_low}. 

This work addresses the following extension of \cref{prob:trad_opt_low}.

\begin{prob}
	\label{prob:nonneg_opt_low}
	Given $N \in \Rmn$, find $M^\opts \in \Rmn$ with $\rk(M^\opts) \leq r$ such that 
	\begin{equation*}
	\min_{\stackrel{M \in \Rmn}{\rk (M) \leq r}}\left[ \frac{1}{2}\|N -M\|_F^2 + g(M)\right] = \frac{1}{2}\|N-M^\opts\|_F^2+g(M^\opts),
	\end{equation*}
	where $g: \Rmn \to \mathbb{R} \cup \lbrace \infty \rbrace$ is a given closed, proper and convex function, i.e., the epi-graph of $g$ is closed, non-empty and convex, respectively.
\end{prob}	
Compared to \cref{prob:trad_opt_low}, \cref{prob:nonneg_opt_low} has an additional function $g$ that can be used to add information about the desired solution. Both problems are non-convex due to the rank constraint. Nevertheless, we will see in~\Cref{sec:main} that they can often be solved by convex optimization and semi-definite programming.

In the following, we often use $g(M) \equiv \chi_{\mathcal{C}}(M)$, where 
\begin{align*}
\chi_{\mathcal{C}}(M) := \begin{cases}
0, & M \in \mathcal{C}\\
\infty, & M\notin \mathcal{C}
\end{cases}
\end{align*}
is defined to be the indicator function of a (convex) set $\mathcal{C} \subset \Rmn$. We also use $\chi_{\rk(M)\leq r}$ to denote the indicator function of the set of matrices with at most rank $r$. In the remainder of this paper, it is assumed that $g + \chi_{\rk(M)\leq r}$ is proper. 

\subsection{Nuclear Norm Regularization}
\label{subsec:nuc:_reg}

One of the most widely used methods to convexify rank constrained problems is to use nuclear~norm regularization. It borrows techniques from sparse regularized regression (see~\cite{hastie2015statistical}), where the $\ell_1$ norm is used as a sparsifier.

In our case, rather than having a sparse solution, we are interested in having a small number of non-zero singular values. The nuclear norm imposes an $\ell_1$ norm penalty on the singular values. Therefore, for given $N \in \Rmn$, a matrix version for convexifying \cref{prob:nonneg_opt_low} reads
\begin{equation}
\label{eq:nuc_heu}
\min_{M \in \Rmn} \frac{1}{2}\|N -M\|^2_F + \mu \|M\|_{1\ast} +g(M),
\end{equation}
where $g: \Rmn \to \mathbb{R} \cup \lbrace \infty \rbrace$ is a closed and proper convex function. The simplicity of this convexification, as well as the results in \cite{fazel2001rank,recht2010guaranteed}, stimulated a large growth in the application of this method. However, it is often challenging to choose $\mu$ a priori in order to obtain a solution of specific rank. Commonly one assumes that the rank as a function of $\mu$ looks like a staircase, i.e., a large/small $\mu$ decreases/increases the rank. 

In general, this heuristic does not return an optimal solution to \cref{prob:nonneg_opt_low}. In particular, in the case $g = 0$, one usually cannot choose $\mu$ such that the SVD-approximation is obtained. Finally, there is no certificate for checking whether a solution is a minimizer of \cref{prob:nonneg_opt_low}.

\section{The \texorpdfstring{$r\ast$}{r*} approach}
\label{sec:main}
In the following, we consider the problem of finding solutions to \cref{prob:nonneg_opt_low}.  
Our approach is based on convex relaxations of \cref{prob:nonneg_opt_low} by means of what we call the $r\ast$ norms. These norms are defined in the following lemma.
\begin{lem}\label{lem:norm}
	Let $M \in \Rmn$, and $r \in \{1,\dots, n\}$. Then,
	\begin{equation}
	\| M \|_r := \sqrt{\sum_{i=1}^r \sigma_i^2(M)} = \sup_{\stackrel{\|X\|_F = 1}{\rk(X) \leq r}} \langle M,X\rangle \label{eq:firstass}
	\end{equation} 
	is a unitarily invariant~norm with dual~norm
	\begin{align*} 
	\| M \|_{r\ast} &:= \max_{\|X \|_r \leq 1}  \langle M,X \rangle = 
	\max_{\sum_{i=1}^r s_i^2 \leq 1} \left[ \sum_{i=1}^{r} \sigma_i(M) s_i + s_r \sum_{i=r+1}^{n} \sigma_i(M) \right].
	\end{align*} 
	Moreover, 
	\begin{align}
	&\|M\|_1 \leq \dots \leq \|M\|_n = \|M \|_F = \|M\|_{n\ast} \leq \dots \leq \|M\|_{1 \ast}, \label{eq:norm_ineq}\\
	&\rk(M)  \leq r \text{ if and only if } \|M\|_r = \|M\|_F = \|M \|_{r\ast}. \label{eq:rank_norm}
	\end{align}
\end{lem}
A proof to this lemma is provided in~\Cref{subsec:proofs}. Notice that $\| M \|_1 = \sigma_1(M)$ is equal to the spectral~norm and its dual~norm $\| M \|_{1\ast} = \sum_{i=1}^{n}\sigma_i(M)$ is equal to the nuclear (trace~norm). These~norms can be formulated using convex linear matrix inequalities (see~\cite{fazel2001rank,recht2010guaranteed}). In~\Cref{sec:main} it is shown that the same holds true for $\| \cdot \|_r^2$ and $\| \cdot \|_{r\ast}^2$. 

Next we show that the $r\ast$ norm can be used to construct the largest convex minorizer (convex envelope) of
$$f(M) := \frac{1}{2}\|N-M\|^2_F + \chi_{\rk(M)  \leq r}(M).$$

\begin{thm}\label{thm:con_bicon}
	Let $N \in \Rmn$, and $r \in \{1,\dots, n\}$. Then the conjugate and bi-conjugate functions of $$f(M) := \frac{1}{2}\|N-M\|^2_F + \chi_{\rk(M)  \leq r}(M)$$ are given by
	\begin{align} 
	f^\ast(D) &= \frac{1}{2}\|N+D\|_r^2 - \frac{1}{2}\| N  \|_F^2, \label{eq:fconj}\\
	f^{\ast \ast} (M) &=  \frac{1}{2}\|M\|_{r\ast}^2 - \langle N,M \rangle + \frac{1}{2}\|N\|_F^2 \label{eq:fbiconj}
	\end{align}
	for all $D, M \in \Rmn$.
\end{thm}

A proof to \cref{thm:con_bicon} can be found in~\Cref{subsec:proof_con_bicon}. Note that by Fenchel duality (see~\cite[Section~31]{rockafellar1970convex}) the following Lemma holds.
\begin{lem}
	\label{lem:zero-dual}
	Let $f, g: \funcdom$ be such that $g$ is proper, closed and convex. Then,
	\begin{align}
	\inf_{M \in \Rmn} \left[ f(M) +g(M) \right] &\geq -\inf_{D \in \Rmn}\left[f^{\ast}(D)+g^{\ast}-(D)\right] \label{lem:fenchel_inequ_left}\\ 
	& = \inf_{M \in \Rmn}\left[f^{\ast \ast}(M)+g(M)\right]. \label{lem:fenchel_inequ}
	\end{align}	
	If $M^\opts$ is a solution to \cref{lem:fenchel_inequ} such that $f(M^\opts) = f^{\ast \ast}(M^\opts)$, then $M^\opts$ is also a solution to the left-hand side of \cref{lem:fenchel_inequ_left}.
\end{lem}

Therefore, we can construct the dual and bi-dual problems to \cref{prob:nonneg_opt_low} as
\begin{align}
-&\min_{D \in
	\Rmn} \left[g^\ast(-D) + \frac{1}{2}\|N+D\|_r^2 - \frac{1}{2}\| N  \|_F^2 \right], \label{eq:dual_problem} \tag{A}\\
&\min_{M \in \Rmn} \left[ \frac{1}{2}\|M\|_{r\ast}^2 - \langle N,M \rangle + \frac{1}{2}\|N\|_F^2+g(M)\right],\tag{B} \label{eq:dualdaul}
\end{align}
which are accompanied by the next central result. 
\begin{prop}\label{prop:main}
	Let $N \in \Rmn$ and $g: \Rmn \to \mathbb{R} \cup \lbrace \infty \rbrace$ be a closed proper convex function. Then for all $r \in \{1,\dots, n \}$
	\begin{align}
	\min_{ \substack{M\in \Rmn \\\rk(M) \leq r}}\left[\frac{1}{2}\|N-M\|^2_F +g(M)\right] \notag
	& \; \geq \; -\min_{D \in
		\Rmn} \left[g^\ast(-D) + \frac{1}{2}\|N+D\|_r^2 - \frac{1}{2}\| N  \|_F^2 \right]	\label{eq:duality_gap} \tag{C}\\ 
	&= \; \min_{M \in \Rmn} \left[ \frac{1}{2}\|M\|_{r\ast}^2 - \langle N,M \rangle + \frac{1}{2}\|N\|_F^2+g(M)\right].\notag
	\end{align}	
	Assume that \cref{eq:dualdaul} has a minimizer $M^\opts$ with $\rk(M^\opts) \leq r$. Then,
	\begin{align*}
	\argmin_{\stackrel{M \in \Rmn}{\rk(M) \leq r}} \left[ \frac{1}{2}\|N-M\|_F^2 + g(M)\right]
	\; \subset \; \argmin_{{M \in \Rmn}} \left[ \frac{1}{2}\|M\|_{r\ast}^2 - \langle N,M \rangle + \frac{1}{2}\|N\|_F^2+g(M)\right].\label{prop:dualdual_primal_opt}
	\end{align*}
\end{prop}
Thus obtaining a $\rk$-$r$ solution to the convex relaxation problem \cref{eq:dualdaul} implies solving the original non-convex problem.  This is why we suggest to use \cref{eq:dualdaul} instead of the the nuclear~norm heuristic (see \cref{eq:nuc_heu} in \Cref{subsec:nuc:_reg}) as convex relaxation to \cref{prob:nonneg_opt_low}.
Nevertheless, in general there may be a duality-gap for some choices of $g$ (see~\cref{sec:comp}). This is reflected by the inequality in \cref{eq:duality_gap}. Fortunately, there are many situations with no duality-gap. Next, an important case is discussed to provide additional insights.

\begin{prop}
	\label{prop:D_uniqueness}
	Assume that $D^\opts$ is a solution to \cref{eq:dual_problem} and
	$\sigma_r(N+D^\opts) \neq \sigma_{r+1}(N+D^\opts)$ or $\sigma_r(N+D^\opts) = 0$. Then there is no duality gap in \cref{eq:duality_gap} and $\svd_r(N+D^\ast)$ is the unique minimizing argument of~\cref{prob:nonneg_opt_low}, i.e.
	\begin{equation*}
	\svd_r(N+D^\opts) = \argmin_{\stackrel{M \in \Rmn}{\rk(M) \leq r}} \left[ \frac{1}{2}\|N-M\|_F^2 + g(M)\right].
	\end{equation*}
\end{prop}
\cref{prop:D_uniqueness} provides a simple sufficient condition for the uniqueness of a solution to \cref{prob:nonneg_opt_low}, which in many applications is fulfilled (see~\cref{sec:hankel}). However, this is not a necessary condition. A proof of \cref{prop:D_uniqueness} is given in a more general setting in~\cref{prop:DM_relation}, which also allows us to say something about the rank of the solution to the convex relaxation if there is a duality-gap.
\begin{thm}
	Let $D^\opts$ and $M^\opts$ be solutions to \cref{eq:dual_problem} and \cref{eq:dualdaul}, respectively. Further, suppose that an SVD of $N+D^\opts$ is given by $N+D^\opts = \sum_{i=1}^{n} \sigma_iu_i v_i^{\transp}$ with 
	$\sigma_r = \dots = \sigma_{r+s} \neq \sigma_{r+s+1},$ where $s=n-r$ if $\sigma_{n} = \sigma_{r}$. Then,
	$$M^\opts \in \conv(\svd_r(N+D^\opts)).$$
	In particular, $\rk(M^\opts)\leq r+s$. Moreover, if $\sigma_r \neq \sigma_{r+1}$ or $\sigma_r = 0$, then $M^\opts = \svd_r(N+D^\opts).$
	\label{prop:DM_relation}
\end{thm}
A proof to this theorem is given in~\cref{subsec:proof_DM_relation}.
Observe that whenever \cref{eq:dualdaul} does not have a unique solution, it follows by \cref{prop:DM_relation} that $$\sigma_r(N+D^\opts) = \sigma_{r+1}(N+D^\opts)$$ for all solutions $D^\opts$ to \cref{eq:dual_problem}. Furthermore, \Cref{prop:DM_relation} shows that $\svd_r(N)$ with $\sigma_r(N) \neq \sigma_{r+1}(N)$ can be determined by solving a convex problem.
\begin{cor}
	\label{prop:dual_problem}
	Let $N \in \Rmn$, and $r \in \{1,\dots, n\}$. Then, 
	\begin{align*}
	\min_{ \substack{M\in \Rmn \\\rk(M) \leq r}}\frac{1}{2}\|N-M\|^2_F \; = \;  \frac{1}{2}\| N  \|_F^2-\frac{1}{2}\|N\|_r^2 \; 
	 =  \; \min_{M \in \Rmn} \left[ \frac{1}{2}\|M\|_{r\ast}^2 - \langle N,M \rangle + \frac{1}{2}\|N\|_F^2 \right]	
	\end{align*}
	and $$\svd_r(N) \subset \argmin_{M \in \Rmn} \left[\frac{1}{2}\|M\|_{r \ast}^2-\langle N,M\rangle\right].$$
	If $\sigma_r(N) \neq \sigma_{r+1}(N)$ or $\sigma_{r} =~0$ then 
	\begin{equation*}
	\svd_r(N) = \argmin_{M \in \Rmn} \left[\frac{1}{2}\|M\|_{r \ast}^2-\langle N,M\rangle\right].
	\end{equation*}
\end{cor}
\begin{proof}
Since $g = 0$, $g^\ast(D)$ is finite if and only if $D = 0$. Thus the result follows by \Cref{prop:DM_relation}.
\end{proof}
The low-rank inducing property of the $r\ast$ norm can also be seen by characterizing the extreme points of its unit ball.  
\begin{lem}
	\label{lem:convhull}
	The set of the extreme points of the unit-ball $B_1 := \{X: \normrast{X}\leq 1 \}$ is $$E := \{X \in \Rmn : \|X\|_{F} = 1, \ \rk(X) \leq r \}.$$
	Hence, $B_1 = \conv(E)$.	
\end{lem}
\begin{proof}
	By \cref{eq:firstass} in \Cref{lem:norm}, it holds that for all $N \in \Rmn$ 
	\begin{align}
	\sup_{M \in \conv(E)} \langle N, M\rangle =  \|N\|_r = \sup_{M \in B_{1}} \langle N, M\rangle. \label{eq:lem_norm_cov}
	\end{align}
	Since $\conv(E)$ and $B_{1}$ are closed convex sets, \cite[Corollary~13.1.1.]{rockafellar1970convex} implies that $B_{1} = \conv(E)$. If a point $\bar{M} \in E$ is not an extreme point of $E$, then 
	\begin{equation*}
	\textstyle \bar{M} = \sum_{i} \alpha_i M_i, \quad \text{with} \quad \sum_{i} \alpha_i = 1,
	\end{equation*}
	such that
	$$\ M_i \in K \setminus \lbrace \bar{M} \rbrace \quad \text{and} \quad \alpha_i > 0 \ \text{for all }i.$$
	Hence, by the Cauchy-Schwarz inequality we conclude that 
	\begin{align*}
	\textstyle 1 = \langle \bar{M}, \bar{M} \rangle = \sum_{i} \alpha_i \langle \bar{M},  M_i \rangle \leq \sum_{i} \alpha_i = 1.
	\end{align*}
	However, this can only be true if $\langle \bar{M} , M_i \rangle  = 1$ for all $i$. Equivalently, $\bar{M} = M_i$ and that is a contradiction. 
\end{proof}
Finally, the preceding results cover several extensions of \cref{prob:nonneg_opt_low}. By letting $N = \diag(v)$ and $M = \diag(w)$ for \linebreak $v,w \in \mathbb{R}^n$, there are analogous norms for vector-valued problems, (see~e.g.~\cite{argyriou2012sparse,laurent2009group}) where our analysis carries over. Further, it is possible to consider the weighted case
\begin{align}
\min_{ \substack{M\in \Rmn \\\rk(M) \leq r}}\left[\frac{1}{2}\|W(N-M)\|^2_F +g(M) \right], \label{eq:extend_norm}
\end{align}
where $W \in \mathbb{R}^{l \times n}$ and $\rk(W) = n$. Since $\rk(\tilde{M}) = \rk(W^\dagger \tilde{M}) = \rk(M)$, \cref{eq:extend_norm} can be reformulated such that it fits \cref{prob:nonneg_opt_low} by letting $\tilde{g}(\tilde{M}) := g(W^\dagger \tilde{M})$:
\begin{align*}
\min_{ \substack{M\in \Rmn \\\rk(M) \leq r}} \left[  \frac{1}{2}\|W(N-M)\|^2_F +g(M) \right] \; = \; \min_{ \substack{\tilde{M} \in \Rmn \\\rk(\tilde{M}) \leq r}} \left[\frac{1}{2}\|WN - \tilde{M}\|^2_F + \tilde{g}(\tilde{M})\right].
\end{align*}
Since another inner product and norm is defined by $W$ as
\begin{align*}
\|W(N-M)\|^2_F &= \trace((N-M)^{\transp}W^{\transp}W(N-M)) \\
 &=: \langle N-M,N-M \rangle_{W^{\transp}W},
\end{align*}
a suitable $W$ may enable us to satisfy the requirements of \cref{prop:D_uniqueness} in situations where the Frobenius~norm fails. In particular, $W$ may be used for iterative re-weighting. For vector-valued problems, this generalizes the idea of $\ell_1$ norm re-weighting (see~\cite{candes2008enhancing}) to $r\ast$ norms.

\section{Real-valued Extension} 
\label{subsec:reg-interp}
In the following, it is shown that allowing $r$ to be real-valued can be considered as a regularization parameter. Unlike typical regularization methods (see~\cite{fazel2001rank,larsson2016convex}), this parameter has a close relationship to the rank of the corresponding solutions.

It suffices to discuss the case where \cref{prop:D_uniqueness} does not apply. Therefore, let $$D_t^\opts := \argmin_{D \in
	\Rmn} \left[g^\ast(-D)  + \frac{1}{2}\|N+D\|_t^2 \right],$$
{and}
$$M_t^\opts := \argmin_{M \in \Rmn} \left[\frac{1}{2}\|M\|_{t\ast}^2 - \langle N,M \rangle  +g(M)\right].$$
be defined for all $t \in \{1,\dots n\}$, and assume that there exists $r \in \mathbb{N}$ with $$\sigma_r(N+D_r^\opts) = \sigma_{r+1}(N+D_r^\opts) \ \text{ and } \ \rk(M_r^\opts) > r.$$
Furthermore, let $$\frac{1}{2}\|N-M_r^\opts\|_F^2 +g(M_r^\opts) > \frac{1}{2}\|N-M_{r+1}^\opts\|_F^2 +g(M_{r+1}^\opts)$$ with $$\rk(M_{r+1}^\opts) > \rk(M_{r}^\opts).$$ In such a scenario, one often faces the situation that $\rk(M_r^\opts)$ is small, but the cost $\frac{1}{2}\|N-M_{r}^\opts\|_{F}^2+g(M_r^\opts)$ is poor, whereas $\frac{1}{2}\|N-M_{r+1}^\opts\|_{F}+g(M_{r+1}^\opts)$ may be acceptable, but $\rk(M_{r+1}^\opts)$ is too large. Then a trade-off between $M_r^\opts$ and $M_{r+1}^\opts$ is desired. Such a trade-off can be achieved by letting $r$ become non-integer valued in the $r$~norm. The $r$~norm extends to
\begin{align}
\| \cdot \|_r := \sqrt{\sum_{i=1}^{\lfloor r \rfloor} \sigma_i^2(\cdot) + (r - \lfloor r \rfloor) \sigma_{\lceil r \rceil}^2(\cdot)}, \label{eq:rnorm_real}
\end{align}
where $\lfloor r \rfloor := \max \lbrace z \in \mathbb{Z}: z \leq r \rbrace$ and $\lceil r \rceil := \min \lbrace z \in \mathbb{Z}: z \geq r \rbrace$. For $r \in \mathbb{N}$ and $\alpha \in [0,1]$ we have
\begin{align}
\| \cdot \|_{r+\alpha}^2 = (1-\alpha) \| \cdot \|_{r}^2 + \alpha\| \cdot \|_{r+1}^2, \label{eq:rnorm_conv}
\end{align}
which means that $\| \cdot \|_{r+1-\alpha}^2$ is a convex combination of $\| \cdot \|_{r}^2$ and $\| \cdot \|_{r+1}^2$, and thus indicates its usefulness in supplying the desired trade-off solution. Similar to \cref{prop:DM_relation}, it remains true by \cref{prop:sub_diff_rnorm} that $\rk(M_r^\opts) \leq \lceil r \rceil + s$ if $r\in \mathbb{R}_{\geq 1}$ and
\begin{align}
\sigma_{\lceil r \rceil}(N+D_r^\opts) = \dots = \sigma_{\lceil r \rceil+s}(N+D_r^\opts) > \sigma_{\lceil r \rceil+s+1}(N+D_r^\opts). \label{eq:sigma_eq}
\end{align}
Hence, allowing $r$ to assume real values may allow us to find solutions of both lower rank and lower cost. Next we look at the dependency of $s$ on $r$ in \cref{eq:sigma_eq}. We define
\begin{align*}
F(D,r) &:= g^\ast(-D)  + \frac{1}{2}\|N+D\|_r^2 + \frac{1}{2} \|N\|_F^2.
\end{align*}
Using the piecewise linearity in \cref{eq:rnorm_conv}, it can be shown that $F$ is (jointly) continuous on the relative interior of its domain. Therefore, Berge's Maximum Theorem (see~ \cite[p. 116]{berge1963topological}) implies that the parameter depending set
\begin{align*}
\mathcal{C}^\opts(r) := \argmin_{D \in
	\Rmn} \left[g^\ast(-D)  + \frac{1}{2}\|N+D\|_r^2 + \frac{1}{2} \|N\|_F^2 \right]
\end{align*}
is upper hemicontinuous in $r$. This means that for all \linebreak $r \in [1, \minmn]$ and all $\varepsilon > 0$ there exists $\delta > 0$ such that for all $t \geq 1$
\begin{align}
|t-r| < \delta \Rightarrow \mathcal{C}^\opts(t) \subset \mathcal{B}_{\varepsilon}  \left(\mathcal{C}^\opts(r) \right), \label{eq:neigh}
\end{align}
where $$\mathcal{B}_{\varepsilon}  \left(\mathcal{C}^\opts(r) \right) := \left\lbrace X \in \Rmn :\exists D \in \mathcal{C}^\opts(r) \ \textnormal{with} \ \|X-D\|_F < \varepsilon \right\rbrace.$$
For simplicity assume that $D_r^\opts$ is unique. By \cref{eq:neigh} and the continuity of the singular values (see~\cite[Corollary~4.9]{stewart1990matrix}), it follows that a sufficiently small increase of $r$ does not increase $s$ in \cref{eq:sigma_eq}. Hence, just as for nuclear~norm regularization, $\rk(M_t^\opts)$ often looks like a staircase as $t$ varies over $[r,r+1]$ (see~\cref{fig:hankel_rank} in~\Cref{sec:hankel}).
In summary, real-valued $r$ can be considered as a regularization parameter, similar to other regularization methods such as in~\cite{fazel2001rank,larsson2016convex}. 
\section{SDP-Representations}
\label{sec:comp}
Next we develop SDP-representations of the problems \cref{eq:dual_problem,eq:dualdaul} under the assumption that $g$ is SDP-representable. 
We start with an SDP-representation of the optimization problem
\begin{equation}
\label{eq:SDP_D}
\min_{D \in \Rmn} \|N+D\|_r^2,
\end{equation}
where $\|\cdot \|_r$ is defined as in~\cref{eq:rnorm_real} and $r \in [1,n]$. Let $T \in \Rnn$ be such that 
\begin{equation*}
T \succeq (N+D)(N+D)^{\transp}.
\end{equation*}
Then $\sigma_i(T) \geq \sigma_i^2(N+D)$ for all $i$ such that $1\leq i \leq n$ (see~\cite[Corollary 7.7.4]{horn2012matrix}) and $\trace(T) = \sum_{i=1}^n \sigma_i(T)$. Hence,
\begin{align*}
\|N+D\|_r^2 
&\leq \trace(T) - (n - r) \sigma_{n}(T),
\end{align*}
which implies that
\begin{align}
\|N+D\|_r^2 \leq \min_{T \succeq (N+D)(N+D)^{\transp}}  \trace(T) - (n-r) \sigma_{n}(T). \label{eq:rnorm_T_ineq}
\end{align} 
In particular, equality in~\cref{eq:rnorm_T_ineq} can be achieved with $$T^{\opts} := \sum_{i=1}^{\lceil r \rceil} \sigma_i^2(N+D) u_i u_i^{\transp} + \sigma_{\lceil r \rceil}^2(N+D)\sum_{i=\lceil r \rceil+1}^{n}u_i u_i^{\transp},$$ where $N+D = \sum_{i=1}^{n} \sigma_i(N+D) u_i v_i^{\transp}$ is an SVD of $N+D$. Using the Schur-complement condition for $T - (N+D)(N+D)^{\transp} \succeq0$ (see~\cite[Theorem~7.7.7]{horn2012matrix}) yields that
\begin{equation*}
\begin{aligned}
& {\underset{D,T,\gamma}{\textnormal{minimize}}}
& & \trace(T) - \gamma (n-r)\\
& \textnormal{subject to}
& & \begin{pmatrix}
T & N+D \\
(N+D)^{\transp} & I
\end{pmatrix} \succeq 0, \ T \succeq \gamma I , \
D \in \Rmn.  
\end{aligned}
\end{equation*}
is an SDP-representation for \cref{eq:SDP_D}. Then, an SDP-formulation of \cref{eq:dualdaul} can be obtained by deriving the dual of this optimization problem as 
\begin{equation}
\label{SDP:rast}
\begin{aligned}
& {\underset{M,P,W}{\textnormal{minimize}}}
& & \frac{1}{2}\trace(W) - \trace(N^{\transp}M) +g(M)\\
& \textnormal{subject to}
& & \begin{pmatrix}
I-P & M \\
M^{\transp} & W
\end{pmatrix} \succeq 0, \ P \succeq 0, \ \trace(P) = n-r.
\end{aligned}
\end{equation}

\section{Model Order Reduction}
\label{sec:hankel}
In system and control, the rank of a Hankel matrix/operator is important, because it determines the order, e.g. of a linear time invariant discrete-time system
\begin{equation}\label{eq:lin_sys}
\begin{aligned}
x_{k+1} &= Ax_k + Bu_k,\\
y_k   &= Cx_k + Du_k,
\end{aligned}
\end{equation}
where $A \in \mathbb{R}^{n\times n}$, $B \in \mathbb{R}^{n \times n_u}$, $C \in \mathbb{R}^{n_y \times n}$ and $D \in \mathbb{R}^{n_y \times n_u}$. Note that if $(A,B,C,D)$ is a minimal realization, then $n$ is the order of the system and thus decides how costly it is to simulate or control the system (see,~e.g.~\cite{antoulas2005approximation,ankelhed2011design,kung1987identification}). As a result, the field of model order reduction has emerged \cite{antoulas2005approximation}. Whereas the Adamyan-Arov-Krein theorem \cite{antoulas2005approximation} answers the question of optimal low-rank approximation of infinite dimensional Hankel operators, the finite dimensional case
 \begin{equation}
 \label{prob:hankel}
 \begin{aligned}
 & \underset{M}{\textnormal{minimize}}
 & & \|N-M\|_F^2 \\
 & \textnormal{subject to}
 & & M \in \mathcal{H}, \ \rk(M) \leq r,\\
 \end{aligned}
 \end{equation}	
where $N \in \mathcal{H} := \lbrace H: H \textnormal{ is Hankel} \rbrace$, is still an open problem \cite{blondel2012open}.
The finite dimensional case \cref{prob:hankel} is important, e.g for model approximation or system identification (see \cite{kung1987identification,miller2012identification}), where $N$ is formed through the known or measured impulse response, $h_0 = D, \ h_t = CA^{t-1}B, \ t \geq 1$, of a stable linear system \cref{eq:lin_sys}:
\begin{equation*}
N = H_{k,l+1} := \begin{pmatrix}
h_1 & h_2 & \cdots & h_{l+1}\\
h_{2} & h_{3} & \cdots & h_{l+2}\\
\vdots & \vdots &   & \vdots\\
h_{k}   & h_{k+1} & \cdots & h_{k+l}
\end{pmatrix}.
\end{equation*}
Assuming that $k,l \geq n$, it holds that $\rk(N) \leq n$ and $N$ can be mapped onto a minimal realization of \cref{eq:lin_sys} through Kung's (or Ho-Kalman-Kung) algorithm \cite{kung1987identification,miller2012identification,antoulas2005approximation}. Moreover, also a Hankel structured rank-$r$ approximation $M^\opts$ of $N$ can be mapped by Kung's algorithm onto a linear system $(\hat{A},\hat{B},\hat{C},\hat{D})$ of order $r$ if $\rk \left(M^\opts_{(1:r n_y,1:r n_u)} \right)= r$. The system matrices are derived as
\begin{equation}
\label{eq:red_sys_mat}
\begin{aligned}
\hat{A} &= O^\dagger M^\opts_{(1:k,2:l+1)} R^\dagger, &  \hat{B} &= R_{(1:r,1:n_u)}, \\
\hat{C} &= O_{(1:n_y,1:r)}, &  \hat{D} &= h_0,
\end{aligned}
\end{equation}
where an SVD of $M^\opts_{(1:k,1:l)} = \sum_{i=1}^{r} \sigma_r u_i v_i^\transp$ determines
\begin{align*}
O := \begin{pmatrix}
\sigma_1 u_1  &  \dots & \sigma_r u_r
\end{pmatrix} \quad \textnormal{and} \quad R := \begin{pmatrix}
\sigma_1 v_1  &  \dots & \sigma_r v_r
\end{pmatrix}^\transp.
\end{align*}
Consequently, the impulse response matches $M^\opts$, i.e. it fulfils \begin{align*}
\hat{C} \hat{A}^{t-1} \hat{B} = \begin{cases}
M^\opts_{((t-1)n_y + 1 : t n_y, 1:n_u)}, & 1 \leq t \leq k, \\
M^\opts_{(1:n_y, (t-1) n_u + 1 : t n_u)}, & k+1 \leq t \leq k+l.
\end{cases}
\end{align*}
In the following, we compare the performance of the $r\ast$ approach \cref{eq:dualdaul} and nuclear norm regularization \cref{eq:nuc_heu} to balanced truncation for the minimal system \cref{eq:lin_sys} of order 10 with
\begin{align}
A = \diag(0,0.1,\dots,0.9), \ \ C = B^\transp = (1,\dots,1), \ \ D = 0 \label{sys:ex}
\end{align}
and $N = H_{71,71}$ being the intrinsic Hankel matrix. We use Kung's algorithm to map the Hankel matrix approximations of the convex methods onto systems and compare their $H_\infty$ norm errors (see~\cref{fig:hankel_error}) with balanced truncation. Further, we construct Hankel matrix approximations of $N$ from the balanced truncated models and compare their Frobenius norm errors with those of the convex methods (see~\cref{fig:hankel_error}).
  
By the rank evolution in \cref{fig:hankel_rank}, we can see that the $r\ast$ approach exhibits the expected staircase behaviour as discussed in \Cref{subsec:reg-interp}. Further, it can be observed that there is a zero duality gap for all $r \in \{1,\dots,9\}$. Thus by \cref{prop:main}, the Frobenius norm error in \cref{fig:hankel_error} is the lowest for the $r\ast$ approach. In particular, nuclear norm regularization performs 2 -- 34 times worse than BT, whereas the $r\ast$ approach has about 9 -- 16 \% smaller error than BT. An even stronger error difference reveals for the corresponding system errors, where our method performs 17 -- 39 \% better than BT and the nuclear norm 3 -- 200 times worse. Finally, the complete evolution of the normalized errors for the $r\ast$ approach in~\cref{fig:hankel_rank} shows that good approximations can be achieved for both the system as well as $N$. The small gap between the two errors, which increases with $r$, is due to the large sampling horizon in $N$. The horizon of 141 samples insures that the first 10 singular values in $N$ are close to the Hankel singular values of the system. 

Note that a larger horizon would improve the approximation quality even further. However, our chosen horizon seems to give a good trade-off between error performance and computational cost when solving \cref{SDP:rast} through conventional SDP solvers (see~e.g.~\cite{peaucelle2002user}). Moreover, a smaller horizon seems to mainly affect the quality of higher order approximations. 
Finally, note that our method does not necessarily need to reduce the original system. It could also be used after an initial reducing step through other methods \cite{antoulas2005approximation}. 

An implementation of our example can be found in~\cite{grussler2018github}.

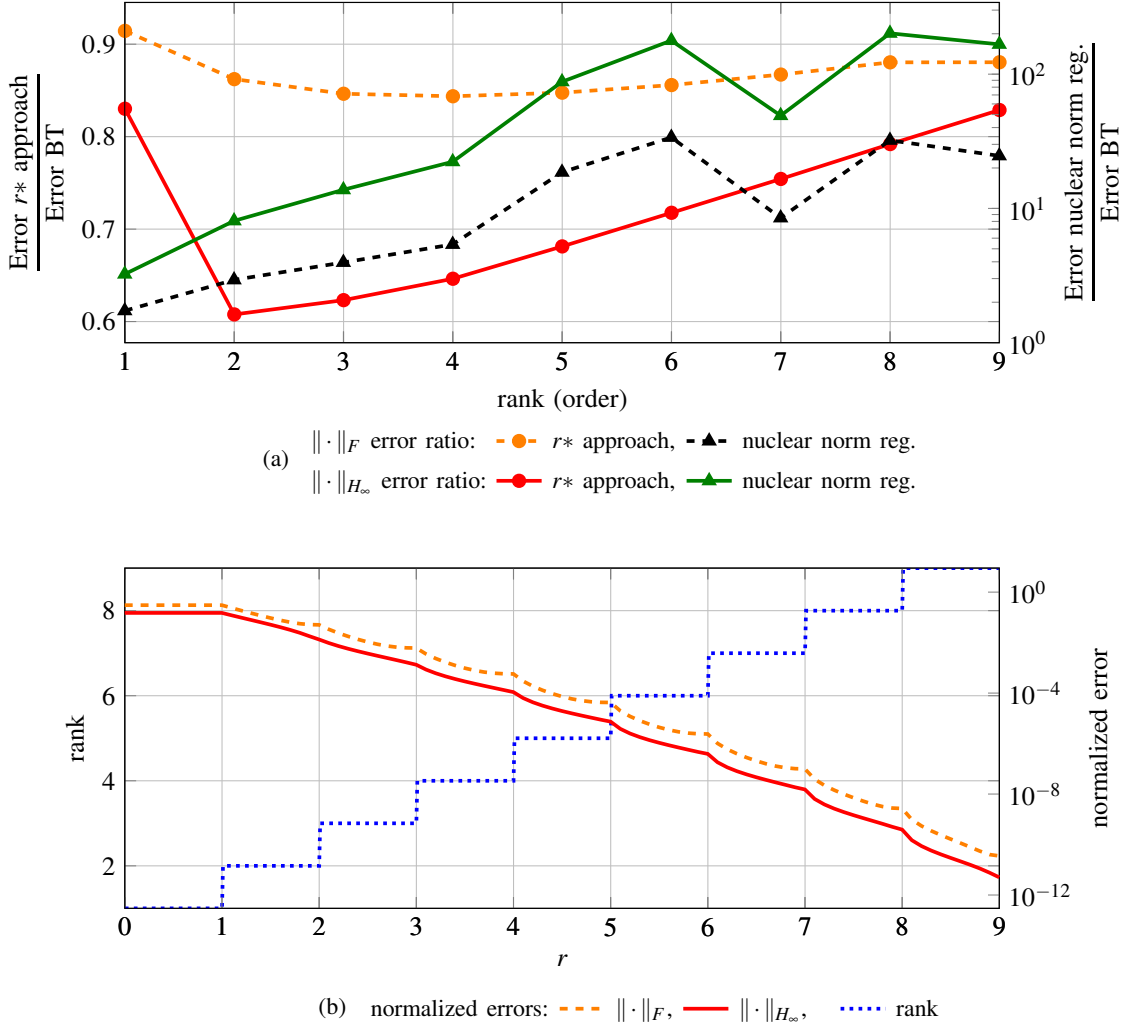
\begin{figure}
	\def\factor{.8}%
 	\def\factorb{.8}
 	\def\factorc{.37}
 	\def\factord{.37}
 		\centering
 	\begin{tikzpicture}

	\begin{groupplot}[group style={group name=my plots, group size=1 by 2,vertical sep = 3.0cm}]
		\nextgroupplot[xlabel=$\textnormal{rank (order)}$,
		ylabel style={align=center},
		ylabel= $\dfrac{\textnormal{Error $r\ast$ approach}}{\textnormal{Error BT}}$,
		xmin = 1,
		xmax = 9,
		grid = both,
		width = \factor \textwidth,
		height = \factord \textwidth,
		xtick = {1,...,9},
		scaled ticks=false, 
		]
		\addplot+[color = FigColor1b,mark = *, mark options={solid},dashed] file{MOR_err_mat_rstar_nom.txt};
		\label{mor_err_mat:rstar}
		
		\addplot+[color = FigColor1,mark = *, mark options={solid}] file{MOR_err_sys_rstar_nom.txt};
		\label{mor_err_sys:rstar}
		\label{mor_err_sys:base}
	
	\nextgroupplot[xlabel=$r$,	ylabel=$\rk$,
	xmin = 0,
	xmax = 9,
	grid = both,
	width = \factorb \textwidth,
	height = \factorc \textwidth,
	ymin = 1,
	ymax = 9,
	xtick ={0,1,2,...,10},
	ytick = {2,4,...,8},
	scaled ticks=false, 
	tick label style={/pgf/number format/fixed}]
	
	\addplot+[color = blue,mark = none,dotted] file{MOR_rank_rstar.txt};
	\label{mor_rank:rast}
	\end{groupplot}

	\begin{groupplot}[group style={group name=my plots, group size=1 by 2,vertical sep = 3.0 cm}]
	\nextgroupplot[xmin = 1,yticklabel pos = right,ytick pos = right,
	ylabel style={align=center},
	ylabel = $\dfrac{\textnormal{Error nuclear norm reg.}}{\textnormal{Error BT}}$,
	xmax = 9,
	xmin = 1,
	ymode = log,
	ymin = 1,
	width = \factor \textwidth,
	height = \factord \textwidth,
	xtick = {1,...,9},
	scaled ticks=false, ]
	\addplot+[color = FigColor4,mark = triangle*, mark options={solid}] file{MOR_err_sys_nuc_nom.txt};
	\label{mor_err_sys:nuc}
	
	\addplot+[color = FigColor3b,mark =  triangle*, mark options={solid},dashed] file{MOR_err_mat_nuc_nom.txt};
	\label{mor_err_mat:nuc}

	\nextgroupplot[yticklabel pos = right, ytick pos = right,
	xmin = 0,
	xmax = 9,
	ylabel = normalized error,
	width = \factorb \textwidth,
	height = \factorc \textwidth,
	ymode = log,
	ymax = 1,
	ymax = 9,
	ymode = log,
	scaled ticks=false, 
	tick label style={/pgf/number format/fixed}]
	\addplot+[color = FigColor1,mark = none] file{MOR_err_sys_rstar.txt};
	\label{MOR_err_sys_rel:rast}
	\addplot+[color = FigColor1b,mark = none, dashed] file{MOR_err_mat_rstar.txt};
	\label{MOR_err_mat_rel:rast}
	\end{groupplot}
	\node[text width=0.47 \textwidth ,align=center,anchor=north] at ([yshift=-7mm,xshift = -1 mm]my plots c1r1.south) {\subcaption{\setlength{\tabcolsep}{1.5 pt}
			\begin{tabular}{ll}
			 $\|\cdot\|_F$ error ratio: & 
			\ref{mor_err_mat:rstar} $r\ast$~approach, \ref{mor_err_mat:nuc} nuclear~norm reg.\\
			 $\|\cdot\|_{H_\infty}$ error ratio: &  \ref{mor_err_sys:rstar} $r\ast$~approach, \ref{mor_err_sys:nuc} nuclear~norm reg.\end{tabular}		
			\label{fig:hankel_error}}};%
	\node[text width=0.38 \textwidth ,align=center,anchor=north] at ([yshift=-7mm,xshift = -1 mm]my plots c1r2.south) 
	{\subcaption{\begin{tabular}{ll}
	normalized errors: \ref{MOR_err_mat_rel:rast} $\|\cdot\|_F$, \ref{MOR_err_sys_rel:rast} $\|\cdot\|_{H_\infty}$, &  \ref{mor_rank:rast} rank \label{fig:hankel_rank}
	\end{tabular}}};%
	\end{tikzpicture}
		\caption{Model Order Reduction for \cref{sys:ex} -- (a) Frobenius norm error of the best achievable approximations to $N = H_{71,71}$ in \cref{prob:hankel} through the $r\ast$ approach \cref{eq:dualdaul} (left y-axis in linear scale) as well as the nuclear norm regularization \cref{eq:nuc_heu} (right y-axis in log scale); $H_\infty$ norm error of the corresponding reduced order systems \cref{eq:red_sys_mat} resulting from Kung's algorithm; both errors are divided by the analogous errors of the balanced truncated (BT) models. (b) Evolution of the rank (left y-axis) as well as the normalized Frobenius and $H_\infty$ norm errors (right y-axis in log scale) of the $r \ast$ approach with real-valued $r \in (0,9]$.}
\end{figure}

\section{Conclusion}
\label{sec:conclusion}
In this work, a method for determining Frobenius norm optimal low-rank approximations with convex constraints has been studied. The main benefits of our approach are  that it is essentially regularization parameter free and may give a certificate of optimality. Moreover, we have seen that our approach can be turned into a regularization dependent method, where, unlike other approaches, the parameter has a direct relationship to the desired rank (see \cref{subsec:reg-interp}).
The model reduction example shows the superiority of our approach over the nuclear-norm heuristic as well as balanced truncation. In addition, our approach allows us to impose further convex constraints onto the impulse response. In the future, we would like to investigate the distinct properties of such approximations, e.g. error bounds, as well as their effectiveness in system identification problems. Furthermore, it would be interesting to see how system characteristics effect a possible duality gap.  
Most of our results can be extended to Hilbert-Schmidt operators. In case of Hankel operators, the singular values translate to Hankel singular values. Nevertheless, our results cannot be easily extended to other unitarily invariant norms, e.g. the spectral norm. This is because other norms often lack the following properties: (1) The norm of a difference of two matrices is not decomposable such that the convex envelope can be easily derived (see~\cite{grussler2016lowrank} for more details), (2) If the norm does not depend on all singular values, then even for $g = 0$ there are (infinitely) many solutions and thus minimizing its convex envelope would almost certainly result in a high rank convex combination of these solutions. Similar effects can be expected for $g \neq 0$. 

Finally note that our approach can also be used to numerically evaluate the performance of heuristics where no relationships to the optimal solutions are known.


%

\appendix
\subsection{Subdifferentials}
\label{subsec:subdiff}
Let $f: \Rmn \to \mathbb{R}\cup
\{\infty\}$ be a convex function, then the subdifferential of $f$ in $X\in \Rmn$ is defined as
\begin{equation*}
\partial f(X) := \{Z: f(Y) \geq f(X) - \langle Y-X,Z\rangle \text{ for all } Y \}.
\end{equation*}
The following proposition on the subgradiential of $\| \cdot \|_r$ has been shown in \cite{doan2016finding} for $r \in \mathbb{N}$. It is straightforward to extend it to the real-valued case.
\begin{prop}
		\label{prop:sub_diff_rnorm} 
		Let $A \in \Rmn \setminus \{0\} $, $r \in [1,n]$ and $\bar{r} := \lceil r \rceil$. Further, let an SVD of $A$ be given by  $A =  \sum_{i=1}^{n} \sigma_i u_i v_i^{\transp}$ with $$\sigma_{\bar{r}-t} \neq \sigma_{\bar{r}-t+1} = \dots = \sigma_{\bar{r}} = \dots = \sigma_{\bar{r}+s} \neq \sigma_{\bar{r}+s+1},$$ where $t = \bar{r}$ and $s=n-\bar{r}$ if $\sigma_1 = \sigma_{\bar{r}}$ and $\sigma_{n} = \sigma_{\bar{r}}$, respectively. Then $M \in \partial \|A\|_{r}$ if and only if
		\begin{align*}
		M &= \dfrac{1}{\|A\|_{r}}\left(\sum_{i=1}^{\bar{r}-t} \sigma_iu_i v_i^{\transp}+ \sigma_{\bar{r}}R \right),\\
		R &= \begin{pmatrix}
		u_{\bar{r}-t+1} & \dots & u_{\bar{r}+s}
		\end{pmatrix} T \begin{pmatrix}
		v_{\bar{r}-t+1} & \dots & v_{\bar{r}+s}
		\end{pmatrix} ^{\transp},
		\end{align*}
		where $T \succeq 0, \quad \|T\|_{1\ast} = t- \bar{r} +  r, \quad \text{and} \quad \|T\|_{1} \leq 1$. 
		Moreover,
		\begin{equation*}
		\partial \|0\|_{r} = \{M \in \Rmn:  \|M\|_{r\ast} \leq 1\}.
		\end{equation*}
	\end{prop}
It is readily seen that \cref{prop:sub_diff_rnorm} is equivalent to $$\partial \normrast{A} = \frac{1}{\normr{A}} \conv(\svd_r(A)).$$
	\subsection{Proof of \texorpdfstring{\cref{lem:norm}}{Lemma~\ref{lem:norm}}}
	\label{subsec:proofs}
	\begin{proof}
		Let $1\leq r \leq n$, $M \in \Rmn$ and the function $g:\mathbb{R}^n \to \mathbb{R}_{\geq 0}$ be defined by
		$$g(x_1,\dots,x_n) := \|\diag(x_1,\dots,x_n)\|_r.$$ 
		The unitary invariance of $\| \cdot \|_r$ follows by \cite[Theorem~7.4.7.2.]{horn2012matrix}, because $g$ is a symmetric gauge function. 
By~\cite[Corollary~7.4.1.3.]{horn2012matrix} it holds that
\begin{align*}
\sup_{\stackrel{\|X\|_F = 1}{\rk(X) \leq r}} \langle X,M\rangle =  \sup_{\sum_{i=1}^{r} \sigma_i^2(X) = 1} \sum_{i=1}^{r} \sigma_i^2(X)\sigma_i(M) = \|M\|_r.
\end{align*}
Then the $r\ast$-norm inherits the unitary invariance of the $r$-norm and with $\Sigma := \diag(\sigma_1(M),\dots,\sigma_{n}(M))$ it follows that
		\begin{align*}
		\| M \|_{r\ast}\;  = \; \|\Sigma \|_{r\ast} \;= \; \max_{\|X \|_r \leq 1} \langle \Sigma,X \rangle \; = \; \max_{\sum_{i=1}^r\sigma_i^2(X)=1} \sum_{i=1}^{n} \sigma_i(M) \sigma_i(X) \;
		= \; \max_{\sum_{i=1}^r\sigma_i^2(X) \leq 1} \left[ \sum_{i=1}^{r} \sigma_i(M) \sigma_i(X) + \sigma_r(X) \sum_{i=r+1}^{n} \sigma_i(M) \right].
		\end{align*} 
		The third equality follows by \cite[Corollary~7.4.1.3.]{horn2012matrix}. Hence, 
		\begin{align*} 
	\|M\|_{1 \ast} \; = \; \max_{\sum_{i=1}^1 s_i^2=1} \sum_{i=1}^{n} \sigma_i(M) s_i \geq \; \dots \; \geq \; \max_{\sum_{i=1}^n s_i^2=1} \sum_{i=1}^n \sigma_i(M) s_i = \|M\|_{n\ast} = \|M\|_F.
		\end{align*}
		Moreover, by the definition of the $r$-norm
		\begin{align*}
		\|M\|_F = \|M\|_{n} \geq \dots \geq \|M\|_1
		\end{align*}
		and therefore \cref{eq:norm_ineq} is shown. In particular, 
		\begin{align*} 
		\|M\|_{r \ast} \; = \; \max_{\sum_{i=1}^r s_i^2=1} \sum_{i=1}^{n} \sigma_i(M) s_i \geq \|M\|_F \; \geq \; \max_{\sum_{i=1}^r s_i^2=1} \sum_{i=1}^{r} \sigma_i(M) s_i \; = \; \|M\|_r. \label{eq:proof_rank}
		\end{align*} 
		Obviously, $\|M\|_F = \|M\|_r$ if and only if $\rk(M) \leq r$, and thus $\|M\|_{r \ast} = \|M\|_{r}$ if and only if $\rk(M) \leq r$.
	\end{proof} 
\subsection{Proof of \texorpdfstring{\cref{thm:con_bicon}}{Theorem~\ref{thm:con_bicon}}}
\label{subsec:proof_con_bicon}
\begin{proof} The conjugate function satisfies
	\begin{align*}
	f^\ast(D) &= \sup_{\substack{M \in \Rmn \\ \rk(M) \leq r}}\left[\langle D,M \rangle - \frac{1}{2}\|N-M\|_F^2 \right]\\
	&= \sup_{\substack{M \in \mathbb{R}^{n \times m} \\ \rk(M) \leq r}}-\dfrac{1}{2}\|N-M + D\|^2_F +  \langle D, N \rangle + \dfrac{1}{2}  \|D \|_F^2\\
	&= -\dfrac{1}{2} \|N + D\|^2_F + \dfrac{1}{2} \|N + D\|^2_r  + \langle D, N \rangle + \dfrac{1}{2} \|D \|_F^2\\
	&=  -\frac{1}{2}\|N\|_F^2 + \frac{1}{2}\|N + D\|^2_r,
	\end{align*}
	where the third equality follows by \cref{prop:Schmidt-Mirsky}. 
	Hence,
	\begin{align*}
	f^{\ast \ast}(M) &= \sup_{D \in \Rmn} \left[\langle D,M \rangle +\frac{1}{2}\|N\|_F^2 - \frac{1}{2}\|N + D\|^2_r \right]\\
	&= \sup_{D \in \Rmn} \left[\langle D-N,M \rangle +\frac{1}{2}\|N\|_F^2 - \frac{1}{2}\|D\|^2_r \right]\\
	&= \frac{1}{2}\|N\|_F^2-\langle N,M \rangle+\sup_{D \in \Rmn} \left[\langle D,M \rangle - \frac{1}{2}\|D\|^2_r \right]\\
	&= \frac{1}{2}\|N\|_F^2 - \langle N, M \rangle + \frac{1}{2} \|M\|_{r\ast}^2,
	\end{align*}
	where the last equality follows by \cite[Corollary~15.3.1]{rockafellar1970convex} with
	\begin{align*}
	\frac{1}{2} \| \cdot \|_{r\ast}^2 = \left( \frac{1}{2} \|\cdot\|_{r}^2 \right)^\ast. 
	\end{align*} 
\end{proof}

\subsection{Proof of \texorpdfstring{\cref{prop:DM_relation}}{Theorem~\ref{prop:DM_relation}}}
\label{subsec:proof_DM_relation}
\begin{proof}
	If $D^\opts$ and $M^\opts$ are solutions to \cref{eq:dual_problem} and \cref{eq:dualdaul}, respectively, then by \cite[Theorem~31.1]{rockafellar1970convex} it holds that $$f^{\ast \ast }(M^\opts) = \langle D^\opts, M^\opts \rangle - f^\ast(D^\opts),$$ where 
	$f^\ast$ and $f^{\ast \ast}$ are given by \cref{eq:fconj} and \cref{eq:fbiconj}. Hence, by \cite[Theorem~23.5.]{rockafellar1970convex} it follows that 
	\begin{align*}
	M^\opts \in \left. \partial_D \frac{1}{2} \|N+D\|_r^2 \right|_{D = D^\opts} = \left. {\|N+D^\opts\|_r}   \partial_D \|N+D\|_r \right|_{D = D^\opts}
	\end{align*}
	and invoking \cref{prop:sub_diff_rnorm} proves the result.
\end{proof}

\section*{Acknowledgment}
The authors would like to thank Andrey Ghulchak for his useful comments and numerous counter-examples. All authors are members of the LCCC Linnaeus Center and the eLLIIT Excellence Center at Lund University. The first author is financially supported by the Swedish Research Council through the project 621-2012-5357. The first and third authors are financially supported by the Swedish Foundation for Strategic Research.

\ifCLASSOPTIONcaptionsoff
  \newpage
\fi



\bibliographystyle{IEEEtran}
\bibliography{refopt,refpos}
\end{document}